\documentclass[12pt, twoside]{article}
\usepackage{amsfonts}
\usepackage{amsmath,amssymb}
\usepackage{multicol}
\usepackage{color}

\newtheorem{theorem}{Theorem}[section]
\newtheorem{lemma}{Lemma}[section]
\newtheorem{corollary}[theorem]{Corollary}

\newtheorem{definition}{Definition}[section]
\newtheorem{remark}{Remark}[section]

\numberwithin{equation}{section}
\newenvironment{proof}[1][Proof]{\noindent\textbf{#1.} }{\hfill $\Box$}
\allowdisplaybreaks

 \makeatletter
\begin{document}
\author{Long-Jiang Gu,\ \  Hong-Rui Sun$\thanks{Corresponding
author}$
 \thanks{ Supported by the program for New Century Excellent Talents in
University(NECT-12-0246) and FRFCU(lzujbky-2013-k02).}
 \\
 \\
{\small  School of Mathematics and Statistics, Lanzhou University,}\\
{\small Lanzhou, Gansu 730000, P.R. China }\\
}
\title{\textbf{\Large Infinitely many small energy solutions of a semilinear $\mathbf{Schr\ddot{o}dinger}$ equation
  }} \maketitle

\begin{abstract}
In this paper we prove the existence of infinitely many small energy solution of a semilinear $Schr\ddot{o}dinger$ equation via
the dual form of the generalized fountain theorem. This equation is with periodic potential and concave-convex nonlinearities.
\end{abstract}

\section{Introduction}
\noindent

 In recent years, strongly indefinite problems have attracted many
 authors' attention. Early in 1998, Kryszewski and Szulkin built the
 generalized linking theorem which is a powerful tool to study the strongly indefinite problems,
 see chapter 6 of \cite{MW} and \cite{SK}.
Using the similar method Batkam and Colin built the generalized fountain theorem and its dual form
 in\cite{BC1,BC2,BC4,BC5} which may be used to find infinitely many large and small energy solutions of
 strongly indefinite problems.

In \cite{He} (see\cite{G}as well)£¬ Barstch and Willem firstly studied the elliptic equation with concave and convex nonlinearities and proved
the exintence of infinitely many small energy solutions via the dual fountain theorem. In \cite{BC4} the authors discussed
the strongly indefinite elliptic systems with concave and convex nonlinearities defined on a bounded domain. A natural question is if we can get
similar results to $Schr\ddot{o}dinger$ equation with periodic potential. At this time the problem may be
strongly indefinite and there does not exist the embedding from $H^{1}(\mathbb{R}^{N})$ to $L^{q}(\mathbb{R}^{N})$
when $1<q<2$. So we discuss the following equation
\begin{equation}\label{P}
      \left\{
   \begin{array}{ll}
     -\Delta u + V(x)u=g(x)| u|^{q-2}u-h(x)| u|^{p-2}u,\\
     u\in H^{1}(\mathbb{R}^{N}) , N\geq3 ,
   \end{array}
   \right.
\end{equation}
whose nonlinearity is with a weight and the weight is of appropriate
attenuation such that it may bring us the embeddings we need. 

In quantum mechanics, the $Schr\ddot{o}dinger$ equation is used to depict the motion law of microscopic particles. The nonlinearity of the $Schr\ddot{o}dinger$ equation means the interaction of two particles. And $g(x)>0$ means
the interaction of two particles mainly acts as attraction when the energy of the particles is small. The weight $g(x)$ looks like a permittivity which means the medium in the space is not so well-distributed and thus leads to that the attraction between two particles becomes weak when they are far away from the origin. The existence of nontrivial solutions shows that these two particles will concentrate at where they can attract each other and form a stable state.

In both \cite{SK} and \cite{BC2} the authors demanded the
nonlinearity to be 1-periodic as well, this condition make the energy functional to be
invariant under the action of group $\mathcal{Z}_{N}$ thus is of benefit to the proof of the nontriviality of solutions.
However, in our work we can not expect the nonlinearity with such a weight to be periodic and we get the
nontriviality of solutions through the $(PS)_{c}$ condition if the weight $h(x)$ attenuates quickly enough so that it provides us the
 compact embedding. If $h(x)$ attenuates slowly we use a variant concentrating-compactness lemma noticing that the solution can not concentrate at $\infty$ when the weight vanishes at $\infty$.


\section{Preliminary}
\indent

In this section we first introduce the abstract critical point theorems which we will need. 

  Let Y be a closed subspace of a separable Hilbert space X
endowed with the usual inner product $(\cdot)$ and the associated norm $\| \cdot \|$.
 We denote by P :$ X$ $\longrightarrow$ $Y$ and  $Q : X \longrightarrow$ $Z =   Y^{\perp}$ the
 orthogonal projections.

  We  fix  an  orthonormal  basis $\{e_{j} \}_{j\geq0}$  of  $Y$ and
  an  orthonormal  basis $\{f_{j} \}_{j\geq0}$  of  $Z$,
 and we consider on $X$ = $Y \oplus Z$   the $\tau$- topology
introduced by Kryszewski and Szulkin in \cite{SK} that is, the topology associated to the following
norm
\begin{center}
    $\| u \|_{\tau}:=
      max\{\sum_{j=0}^{\infty}\frac{1}{2^{j+1}}|(Pu,e_{j})|,
      \| Qu \|\}$ ,  $u\in X$ .
\end{center}
It is easy to see that if $\{u_{n}\}\subset X$ is a bounded sequence, then
 \begin{center}
$u_{n}\rightarrow u$ in $\tau$- topology $\Leftrightarrow$
 $Pu_{n}\rightharpoonup Pu$ and $Qu_{n}\rightarrow Qu$.
 \end{center}

For readers' convenience, we recall the following well-known definitions.

\begin{definition}
Let $\varphi \in C^{1}(X,\mathbb{R})$ and $c \in \mathbb{R}$.

\begin{enumerate}
  \item $\varphi$ is $\tau$-upper$($resp.$\tau$-lower$)$semicontinuous if
for every $C\in\mathbb{R}$ the set \{$u\in X$;$\varphi(u)\geq C$\}$($resp.\{$u\in X$;$\varphi(u)\leq C$\}$)$ is $\tau$-closed.
  \item $\varphi'$ is weakly sequentially continuous if the sequence $\{\varphi'(u_{n})\}$
converges weakly to $\varphi'(u)$ whenever $\{u_{n}\}$ converges weakly to u in X.
  \item $\varphi$ satisfies the Palais-Smale condition((PS) condition for short) if any
sequence $\{u_{n}\}\subset X$ such that $\{\varphi(u_{n})\}$ is bounded and $\varphi'(u_{n})\rightarrow 0$,
has a convergent subsequence.
  \item  $\varphi$ is said to satisfy the Palais-Smale condition at level c
($(PS)_{c}$ condition for short) if any sequence $\{u_{n}\}\subset X$ such that
\begin{center}
     $\varphi(u_{n})\rightarrow c$  and $ \varphi'(u_{n})\rightarrow 0$
\end{center}
has a convergent subsequence.
\end{enumerate}

\end{definition}

Now we introduce the dual form of the generalized fountain theorem
built by Batkam and Colin in \cite{BC4}. This theorem is the key to find small energy solutions.




We adopt the following notations:
\begin{center}
$Y^{k}:=\overline{\bigoplus_{j=k}^{\infty}\mathbb{R}e_{j}}$
and
$Z^{k}:=(\bigoplus_{j=0}^{k}\mathbb{R}e_{j})\bigoplus Z$ .
\end{center}

\begin{theorem}
Let $\varphi \in C^{1}(X,\mathbb{R})$ be an even functional which
is $\tau$-lower semicontinuous and such that $\varphi'$ is weakly sequentially continuous.
If there exist a $k_{0}>0$ such that for every $k\geq k_{0}$ there exists $\sigma_{k}>s_{k}>0$ such that:
\\$(B_{1})$ $a^{k}:=inf_{u\in Y^{k},\|u\|=\sigma_{k}}\varphi(u)\geq 0$,
\\$(B_{2})$ $b^{k}:=sup_{u\in Z^{k},\|u\|=s_{k}}\varphi(u)<0$,
\\$(B_{3})$ $d^{k}:=inf_{u\in Y^{k},\|u\|\leq \sigma_{k}}\varphi(u)\rightarrow 0,k\rightarrow\infty$.
\\Then there exists a sequence $\{u_{k}^{n}\}$ such that
\begin{center}
    $\varphi'(u_{k}^{n})\rightarrow 0$ and $\varphi(u_{k}^{n})\rightarrow c_{k}$ as $n\rightarrow\infty$,
\end{center}
where $c_{k}\rightarrow0$.
\end{theorem}

In order to apply these abstract theory to elliptic systems restrict to a bounded domain $\Omega\subset\mathbb{R}^{n}$ when there is the compact embeddings,
Batkam and Colin introduced the theorems with $(PS)$ condition
 \cite{BC4} Thm 6 .

\begin{corollary}
Under the assumptions of Theorem 2.1, $\varphi$ satisfies in addition:
\\$(B_{4})$ $\varphi$ satisfies the $(PS)_{c}$condition, for all $ c \in [d^{k_{0}},0]$.
Then $\varphi$ has a sequence of critical points $\{u_{k}\}$ such that
$\varphi(u_{k})<0$ and $\varphi(u_{k})\rightarrow 0$ as $k\rightarrow\infty$.
\end{corollary}

  Throughout this paper,  $H^{1}(\mathbb{R}^{N})$ is the standard
 Sobolev space with the norm $\| u \|=\int_{\mathbb{R}^{N}}
 (|u |^{2} + | \nabla u|^{2}) dx$.
 $L^{p}(a(x),\mathbb{R}^{N})$ is the Lebesgue space with positive weight $a(x)$
 endowed with the norm $\| u \|_{L^{p}(a(x),\mathbb{R}^{N})}:=
 (\int_{\mathbb{R}^{N}}| u |^{p} a(x) dx)^{\frac{1}{p}}$.
 By $B(x,r)$ we denote the ball centered at $x$ with radius $r$ .
And the positive constants whose exact value are not important will be denoted by $C$ only.

\section{Main results}
\noindent

In this section,  we discuss the existence of infinitely many small energy solutions of problem (1.1) with the dual form of the generalized fountain theorem and our basic assumptions are:
\begin{description}
  \item[$(H_{1})$] $1<q<2<p<2^{\ast}$, where $2^{\ast}=\frac{2N}{N-2}$.
  \item[$(H_{2})$] The function $V(x):\mathbb{R}^{N}\rightarrow \mathbb{R}$ is continuous and 1-periodic in $x_{1},...,x_{N}$ and 0 lies in a gap of the spectrum of $-\Delta+V$.
  \item[$(H_{3})$] $g\in L^{q_{0}}(\mathbb{R}^{N})\bigcap L^{\infty}(\mathbb{R}^{N})$ with $g(x)>0$, a.e. in $\mathbb{R}^{N}$,
where $q_{0}=\frac{2N}{2N-qN+2q}$.
  \item[$(H_{4})$] $h\in L^{p_{0}}(\mathbb{R}^{N})\bigcap L^{\infty}(\mathbb{R}^{N})$ with $h(x)\geq0$, a.e. in $\mathbb{R}^{N}$,
where $p_{0}=\frac{2N}{2N-pN+2p}$.
  \item[$(H_{4})'$] $h(x)\in L^{\infty}(\mathbb{R}^{N})$ with $h(x)\geq0$, and $h(x)\rightarrow0$ as $|x|\rightarrow\infty$.
\end{description}

\begin{remark}
$(H_{3})$ and $(H_{4})$ mean that the weight of the nonlinearity is of appropriate attenuation so that we can get continuous and  compact embeddings from $H^{1}(\mathbb{R}^{N})$ to the Lebesgue space with weight.

\end{remark}

\begin{remark}
The condition $(H_{4})'$ is weaker than $(H_{4})$. In fact we may construct an example as
$h(x)=\frac{1}{log|x|}$, $|x|>R$, for some $R>0$.

\end{remark}

The nonlinearity of the $Schr\ddot{o}dinger$ equation means the interaction of two particles. And $g(x)>0$ means
the interaction of two particles mainly acts as attraction when the energy of the particles is small.
The weight $g(x)$ looks like a permittivity which means the medium
in the space is not so well-distributed and thus leads to that the attraction between two particles becomes weak when they are
far away from the origin. So our conclusion shows that these two particles will concentrate at where they can attract each other.


Let us introduce the variational setting first and for more information we refer the readers to \cite{SK,BC1,BC2}.

We define a functional $\varphi$ on $H^{1}(\mathbb{R}^{N})$ as
\begin{equation}
 \varphi(u):=\frac{1}{2}\int_{\mathbb{R}^{N}}(  | \nabla u|^{2}+V(x)| u |^{2}) dx
    - \frac{1}{q}\int_{\mathbb{R}^{N}} g(x)| u | ^{q}dx
     +\frac{1}{p}\int_{\mathbb{R}^{N}} h(x)| u | ^{p}dx.
\end{equation}

It is easy to see from condition $(H_{1})$, $(H_{2})$, $(H_{3})$ and $(H_{4})$ or $(H_{4}')$ that $\varphi (u)$ is well defined
and is of class $\mathcal{C}^{1}$, then its critical points are weak solutions of $(3.1)$.
Moreover, any $u\in H^{1}(\mathbb{R}^{N})$ satisfies that 
 if $\phi\in H^{1}(\mathbb{R}^{N})$  then there holds
\begin{equation}
\langle \varphi' (u),\phi\rangle=\int_{\mathbb{R}^{N}}\nabla u\nabla\phi dx
+\int_{\mathbb{R}^{N}}V(x)u \phi dx-\int_{\mathbb{R}^{N}}g(x)| u |^{q-2}u\phi dx
+\int_{\mathbb{R}^{N}}h(x)| u |^{p-2}u\phi dx.
\end{equation}

By $(H_{2})$, the $Schr\ddot{o}dinger$ operator $-\Delta+V(x)$ in $L^{2}(\mathbb{R}^{N})$
has purely continuous
spectrum, and the space $H^{1}(\mathbb{R}^{N})$ can be decomposed into
 $H^{1}(\mathbb{R}^{N})=Y\bigoplus Z$
such that the quadratic form:
\begin{equation}
u\in H^{1}(\mathbb{R}^{N}) \rightarrow \int_{\mathbb{R}^{N}}
 (  | \nabla u|^{2}+V(x)| u |^{2}) dx
\end{equation}
is positive and negative definite on $Y $and $ Z$ respectively and both $Y$ and $Z$
are infinite-dimensional.

Now let $L: H^{1}(\mathbb{R}^{N})\rightarrow H^{1}(\mathbb{R}^{N})$ be the self-adjoint operator
defined by
\begin{equation}
 (Lu,v)_{1}:=\int_{\mathbb{R}^{N}}(\nabla u \nabla v +V(x)uv)dx ,
\end{equation}
where $( \cdot )_{1}$ is the usual inner product in $H^{1}(\mathbb{R}^{N})$.

We denote $P :X\rightarrow Y$ and $Q :X \rightarrow Z $ the orthogonal projections,
and thus we can introduce a new inner product which is equivalent to $(\cdot)_{1}$ by the formula
\begin{center}
    $(u,v):=(L(Qu-Pu),v)_{1}$ , $u,v \in X$
\end{center}
and in this section $\| \cdot\|$ denotes the corresponding norm
\begin{center}
 $\| u \|:=(u,u)^{\frac{1}{2}}$   .
\end{center}
Since the inner products $ (\cdot)$  and $(\cdot)_{1}$ are equivalent,
$Y$ and $Z$ are also orthogonal with respect to $ (\cdot)$.
One can verify easily that $(3.2)$ reads
\begin{equation}
 \varphi(u):=\frac{1}{2}(\| Pu \|^{2}-\| Qu \|^{2})
    - \frac{1}{q}\int_{\mathbb{R}^{N}} g(x)| u | ^{q}dx
     +\frac{1}{p}\int_{\mathbb{R}^{N}} h(x)| u | ^{p}dx.
\end{equation}

In this section, we set
\begin{center}
    $Y_{k}:=\overline{\bigoplus^{\infty}_{j=k}e_{j}}$
and
$ Z_{k}:=[\bigoplus_{j=0}^{k}e_{j}]\bigoplus Z$,
\end{center}
where $\{e_{j}\}_{j\geq0}$ is an orthonormal basis of $(Y,\parallel\cdot\parallel)$.
\\

Our main results in this section are
\begin{theorem}
Assume that the conditions $(H_{1})$,$(H_{2})$,$(H_{3})$ and $(H_{4})$ hold. Then
problem $(1.1)$ has a sequence of nontrival solutions \{$u_{k}$\} with $\varphi(u_{k})< 0$,
and $\varphi(u_{k})\rightarrow 0$ as $k\rightarrow \infty$ .
\end{theorem}

\begin{theorem}
Assume that the conditions $(H_{1})$,$(H_{2})$,$(H_{3})$ and $(H_{4}')$ hold. Then
problem $(1.1)$ has a sequence of nontrival solutions \{$u_{k}$\} with $\varphi(u_{k})< 0$,
and $\varphi(u_{k})\rightarrow 0$ as $k\rightarrow \infty$ .
\end{theorem}

First let us prove the following lemmas.
\begin{lemma}
Assume that $1<q<2^{\ast}$ and $g\in L^{q_{0}}(\mathbb{R}^{N})
\bigcap L^{\infty}(\mathbb{R}^{N})$ with $g(x)\geq 0$ a.e. in $\mathbb{R}^{N}$, where $q_{0}=\frac{2N}{2N-qN+2q}$. Then
$H^{1}(\mathbb{R}^{N})\hookrightarrow L^{q}(g(x),\mathbb{R}^{N}) $
and the embedding is compact.
\end{lemma}
\begin{proof}
For $u\in H^{1}(\mathbb{R}^{n})$, from the $H\ddot{o}lder$ inequality and Sobolev inequality we have
\begin{eqnarray*}
  \int_{\mathbb{R}^{N}}g(x)| u |^{q} dx &\leq& | g(x)|_{L^{q_{0}}}\cdot \left(\int_{\mathbb{R}^{N}}| u |^{\frac{2N}{N-2}} dx\right)^{\frac{qN-2q}{2N}}\\
   &=&   | g(x)|_{L^{q_{0}}}\cdot | u | ^{q}_{L^{2^{\ast}}(\mathbb{R}^{N})} \\
   &\leq&  C \| u \|^{q}.
\end{eqnarray*}
So we have
\begin{equation*}
    | u |_{L^{q}(g(x),\mathbb{R}^{N})}\leq  C \| u \| ,
\end{equation*}
\\which means that
\begin{equation*}
    H^{1}(\mathbb{R}^{N})\hookrightarrow L^{q}(g(x),\mathbb{R}^{N}).
\end{equation*}
Assume that $\{u_{n}\}$ is a bounded sequence in $H^{1}(\mathbb{R}^{n})$,
so it is bounded in $L^{2^{\ast}}(\mathbb{R}^{n})$
and there exists a weak convergent subsequence denoted by $\{u_{n}\}$ also,
according to the Rellich imbedding theorem, when restrict to a bounded domain $\Omega$,
$\{u_{n}\}$ is strongly convergent in $L^{q}(\Omega)$ .

We choose $R>0$ sufficiently large such that for $\varepsilon >0$, there exists  $ M>0 $  such that  if $m,n>M$, we have
\begin{equation*}
    \int_{B(0,R)} | u_{n}-u_{m}|^{q} dx <\frac{\varepsilon}{2| g|_{L^{\infty}}}
\end{equation*}
and
\begin{equation*}
    \left(\int_{\mathbb{R}^{N} \backslash B(0,R)}| g(x) |^{q_{0}} dx\right)^{\frac{1}{q_{0}}}
<\frac{\varepsilon}{4 sup_{n>0} \{| u_{n}|_{L^{2^{\ast}}(\mathbb{R}^{N})} \} }.
\end{equation*}
So,
\begin{eqnarray*}
  \int_{\mathbb{R}^{N}} g(x)| u_{n}-u_{m}|^{q}dx &=& \int_{\mathbb{R}^{N}\backslash B(0,R)} g(x)| u_{n}-u_{m}|^{q}dx
+\int_{B(0,R)} g(x)| u_{n}-u_{m}|^{q}dx \\
   & <& \left(\int_{\mathbb{R}^{N} \backslash B(0,R)}| g(x) |^{q_{0}} dx\right)^{\frac{1}{q_{0}}} | u_{n}-u_{m}|_{L^{2^{\ast}}(\mathbb{R}^{N})}\\
& & + | g|_{L^{\infty}(\mathbb{R}^{N})} \int_{B(0,R)} | u_{n}-u_{m}|^{q} dx  \\
   & <&  \frac{\varepsilon}{2} + \frac{\varepsilon}{2} \\
   & = & \varepsilon .
\end{eqnarray*}
This means that $\{u_{n}\}$ is a Cauchy sequence in $L^{q}(g(x),\mathbb{R}^{N})$, thus complete the proof.
\end{proof}

\begin{lemma}
Under the assumptions of Lemma 3.1, we define:
\\$$\beta_{k}:=sup_{u\in Y_{k},\| u \|=1} | u |_{L^{q}(g(x),\mathbb{R}^{N})} ,$$
\\then  $$\beta_{k}\rightarrow 0  ,   k\rightarrow\infty. $$
\end{lemma}

\begin{proof}
It is clear that $0<\beta_{k+1}\leq\beta_{k}$, so that $\beta_{k} \rightarrow\beta \geq 0$,$k \rightarrow \infty$.
For every $k\geq0$, there exists $u_{k}\in Z_{k}$ such that $\| u_{k} \|=1 $ and $| u_{k}|_{L^{q}(g(x),\mathbb{R}^{N})}>\frac{\beta_{k}}{2}$.
By the definition of $Y_{k}$, $u_{k}\rightharpoonup 0$ in $H^{1}(\mathbb{R}^{N})$. Thus Lemma 3.1 implies that
 $u_{k}\rightarrow 0$ in $L^{q}(g(x),\mathbb{R}^{N})$. So we proved that $\beta_{k}\rightarrow0$ as $k\rightarrow\infty$.
\end{proof}

\begin{lemma}
Under the assumptions $(H_{1})$ $(H_{2})$ $(H_{3})$ and $(H_{4})$ or $(H_{4})'$, the functional $\varphi$ defined in $(3.1)$ (or $(3.5)$) is $\tau-lower$ semicontinuous,
and $\varphi'$ is weakly sequentially continuous.
\end{lemma}

\begin{proof}

 Let $\{{u_{n}}\} \subset X$ and $c\in \mathbb{R}$ such that :
$ u_{n}\rightarrow u$ in $\tau-topology$ and  $\varphi(u_{n})\leq c$.
We write $u_{n}=y_{n}+z_{n}$ , where $y_{n} \in Y ,z_{n} \in Z $.
From the definition of $\tau-topology$, we can see that $z_{n}\rightarrow z$.
\begin{eqnarray*}
  c\geq \varphi(u_{n}) &=& \frac{1}{2}\| y_{n}\|^{2}-\frac{1}{2}\| z_{n}\|^{2}
-\frac{1}{q}| u_{n}|^{q}_{L^{q}(g(x),\mathbb{R}^{N})} + \frac{1}{p}| u_{n}|^{p}_{L^{p}(h(x),\mathbb{R}^{N})}\\
   &\geq & \frac{1}{2}\| y_{n}\|^{2}-\frac{1}{2}\| z_{n}\|^{2}
-\frac{1}{q}| u_{n}|^{q}_{L^{q}(g(x),\mathbb{R}^{N})}.
\end{eqnarray*}
\\ Now we use the Jensen inequality
\begin{eqnarray*}
  | u_{n}|^{q}_{L^{q}(g(x),\mathbb{R}^{N})} &=& | y_{n} + z_{n}|^{q}_{L^{q}(g(x),\mathbb{R}^{N})} \\
   &\leq& ( | y_{n}|_{L^{q}(g(x),\mathbb{R}^{N})}+| z_{n}|_{L^{q}(g(x),\mathbb{R}^{N})})^{q} \\
   &\leq&  2^{q-1} (| y_{n}|^{q}_{L^{q}(g(x),\mathbb{R}^{N})}+| z_{n}|^{q}_{L^{q}(g(x),\mathbb{R}^{N})}),
\end{eqnarray*}
with the embedding $ H^{1}(\mathbb{R}^{N}) \hookrightarrow L^{q}(g(x),\mathbb{R}^{N})$, we have
\begin{eqnarray*}
   c &\geq&  \frac{1}{2}\| y_{n}\|^{2}-\frac{1}{2}\| z_{n}\|^{2}
-\frac{2^{q-1}}{q} (| y_{n}|^{q}_{L^{q}(g(x),\mathbb{R}^{N})}+| z_{n}|^{q}_{L^{q}(g(x),\mathbb{R}^{N})})\\
    &\geq& \frac{1}{2}\| y_{n}\|^{2}-\frac{1}{2}\| z_{n}\|^{2}
- \frac{2^{q-1}}{q} (\| y_{n}\|^{q}+\| z_{n}\|^{q}).
 \end{eqnarray*}
We can see that $\{{y_{n}}\}$ is also bounded, thus $\{u_{n}\}$ is bounded in $H^{1}(\mathbb{R}^{N})$ .
So there exists a subsequence and we also denote it by $\{u_{n}\}$ such that $u_{n}\rightharpoonup u $.
From Lemma 3.1 we have
$$| u_{n}|_{L^{q}(g(x),\mathbb{R}^{n})} \rightarrow | u|_{L^{q}(g(x),\mathbb{R}^{n})},$$
and with the weakly lower semicontinuity of the norm, we have
$$c \geq \varphi (u),$$
thus  $\varphi$  is $\tau-lower$ semicontinuous .

Now let us prove that $ \varphi'$ is weakly sequentially continuous. Assume that
$u_{n}\rightharpoonup u $ in $H^{1}(\mathbb{R}^{N})$.
Then $u_{n}\rightarrow u$ in $L^{2}_{loc}(\mathbb{R}^{N})$
By  $(H_{3})$, $(H_{4})$ and $H\ddot{o}lder$ inequality, we can get that
$\{\varphi'(u_{n})\}$ is bounded, so that  $\varphi'(u_{n}) \rightharpoonup \varphi'(u) $.
\end{proof}

When the condition $(H_{4})'$ holds we can not get the (PS) condition. In order to prove that the solution with nonzero energy must be nontrivial, we need the following
concentration-compactness lemma, and let us show the interpolation inequality for $L^{p}(g(x),\mathbb{R}^{N})$ first.

\begin{lemma}
$($interpolation inequality for $L^{p}(g(x),\mathbb{R}^{N})$ $)$Assume that $1\leq s \leq r\leq t\leq\infty$,
$g \in L^{\infty}(\mathbb{R}^{N})$ with $g(x)\geq 0$ a.e. in $\mathbb{R}^{N}$ and
$\frac{1}{r}=\frac{\theta}{s}+\frac{1-\theta}{t}$.
If $ u \in L^{s}(g(x),\mathbb{R}^{N})\bigcap L^{t}(g(x),\mathbb{R}^{N})$,
then $ u \in L^{r}(g(x),\mathbb{R}^{N})$, and
$$| u|_{L^{r}(g(x),\mathbb{R}^{N})}\leq| u|^{\theta}_{L^{s}(g(x),\mathbb{R}^{N})}
| u|^{1-\theta}_{L^{t}(g(x),\mathbb{R}^{N})}.$$
\end{lemma}

\begin{proof}
Using $H\ddot{o}lder$ inequality, we have
\begin{eqnarray*}
  | u|^{r}_{L^{r}(g(x),\mathbb{R}^{N})}&=& \int_{\mathbb{R}^{N}}g(x)| u|^{r}dx\\
 &=&  \int_{\mathbb{R}^{N}}(g(x)^{\frac{\theta r}{s}}| u|^{\theta r})(g(x)^{\frac{(1-\theta) r}{t}}| u|^{(1-\theta) r})dx\\
   &\leq&  (\int_{\mathbb{R}^{N}}g(x)| u|^{s}dx)^{\frac{\theta r}{s}}
 (\int_{\mathbb{R}^{N}}g(x)| u|^{t}dx)^{\frac{(1-\theta)r}{t}}\\
   &=& | u|^{\theta r}_{L^{s}(g(x),\mathbb{R}^{N})}
| u|^{(1-\theta)r}_{L^{t}(g(x),\mathbb{R}^{N})}.
\end{eqnarray*}
\end{proof}

The idea of the following lemmas come from P.L.Lions (see chapter1 of\cite{MW})and we will prove it completely for readers' convenience.

\begin{lemma}
(concentration-compactness) Let $r>0$ and $2\leq q<2^{\ast}$.If $\{u_{n}\}$ is bounded in $H^{1}(\mathbb{R}^{N})$,
and if $sup_{y\in\mathbb{R}^{N}}\int_{B(y,r)}| u_{n}|^{q} dx \rightarrow 0 $ as $n \rightarrow \infty$, $g\in L^{q_{0}}(\mathbb{R}^{N})\bigcap
L^{\infty}(\mathbb{R}^{N})$ with $g(x)\geq0$,
where $q_{0}=\frac{2N}{2N-qN+2q}$.
Then
$u_{n}\rightarrow 0$ in $L^{p}(g(x),\mathbb{R}^{N})$, for $max\{1,\frac{(N-2)q}{N}\}<p<2^{\ast}$.
\end{lemma}

\begin{proof}
Let $q<s<2^{\ast}$, using the interpolation lemma we have
\begin{eqnarray*}
  | u|_{L^{s}(B(y,r))} &\leq& | u|^{1-\theta}_{L^{q}(B(y,r))} | u|^{\theta}_{L^{2^{\ast}}(B(y,r))} \\
   &\leq& C | u|^{1-\theta}_{L^{q}(B(y,r))}  [\int_{B(y,r)}(| u|^{2}+| \nabla u |^{2}) dx]^{\frac{\theta}{2}},
\end{eqnarray*}
where $\theta=\frac{s-q}{2^{\ast}-q}\frac{2^{\ast}}{s}$. Choosing $s=\frac{2}{\theta}$ and it is easy
to see that this $s$ is valid. So we get
$$\int_{B(y,r)}| u |^{s} dx \leq C | u|^{(1-\theta)s}_{L^{q}(B(y,r))}
\int_{B(y,r)}(| u|^{2}+| \nabla u |^{2}) dx.$$
Now covering $\mathbb{R}^{N}$ by balls of radius $r$, in such a way that each point of $\mathbb{R}^{N}$
is contained in at most $N+1$ balls, we find that
$$\int_{\mathbb{R}^{n}}| u |^{s} dx \leq C(N+1)\int_{\mathbb{R}^{n}}(| u|^{2}+| \nabla u |^{2})dx\cdot  sup_{y\in\mathbb{R}^{n}}\left(\int_{B(y,r)}| u |^{q} \right)^{\frac{(1-\theta)s}{q}}
.$$
Thus
$u_{n}\rightarrow 0$ in $L^{s}(\mathbb{R}^{N})$.
And for $g\in L^{\infty}(\mathbb{R}^{N})$, we have
$u_{n}\rightarrow 0$ in $L^{s}(g(x),\mathbb{R}^{N})$.

For $s<p<2^{\ast}$,by the preceding lemma we have
\begin{equation*}
    | u\mid_{L^{p}(g(x),\mathbb{R}^{n})} \leq | u\mid^{\alpha}_{L^{s}(g(x),\mathbb{R}^{n})}
| u|^{1-\alpha}_{L^{2^{\ast}}(g(x),\mathbb{R}^{N})} <  C | u|^{\alpha}_{L^{s}(g(x),\mathbb{R}^{N})}
| u|^{1-\alpha}_{L^{2^{\ast}}(\mathbb{R}^{N})},
\end{equation*}
where $\alpha=\frac{(2^{\ast}-p)s}{(2^{\ast}-s)p}$.
So we have $u_{n}\rightarrow 0$ in $L^{p}(g(x),\mathbb{R}^{N})$ , when $s<p<2^{\ast}$.

For $max\{1,\frac{(N-2)q}{N}\}<p<s$, we choose $t\in(max\{1,\frac{(N-2)q}{N}\},s)$. Similarly we have
\begin{eqnarray*}
  | u|_{L^{p}(g(x),\mathbb{R}^{N})} &\leq& | u|^{\beta}_{L^{t}(g(x),\mathbb{R}^{N})}
| u|^{1-\beta}_{L^{s}(g(x),\mathbb{R}^{N})} \\
  &=&  \left(\int_{\mathbb{R}^{N}} g(x)| u |^{t} dx\right)^{\frac{\beta}{t}}
| u|^{1-\beta}_{L^{s}(g(x),\mathbb{R}^{n})}\\
  & \leq &  \left(| g(x) |_{L^{q_{0}}} \int_{\mathbb{R}^{N}}| u |^{t\frac{2N}{qN-2q}}dx\right)^{\frac{\beta}{t}}
| u|^{1-\beta}_{L^{s}(g(x),\mathbb{R}^{N})}\\
& \leq &  \left(| g(x) |_{L^{q_{0}}} \| u \|^{\frac{qN-2q}{2Nt}}\right)^{\frac{\beta}{t}}
| u|^{1-\beta}_{L^{s}(g(x),\mathbb{R}^{N})},
\end{eqnarray*}
where $\beta=\frac{(s-p)t}{(s-t)p}.$
So we have $u_{n}\rightarrow 0$ in $L^{p}(g(x),\mathbb{R}^{n})$ , when $max\{1,\frac{(N-2)q}{N}\}<p<s$.
Thus the proof is complete.
\\
\end{proof}

\begin{lemma}
Let $r>0$, $\{u_{n}\}$ is bounded in $H^{1}(\mathbb{R}^{N})$,
$g\in L^{q_{0}}(\mathbb{R}^{N})\bigcap L^{\infty}(\mathbb{R}^{N})$ with $g(x)>0$,
where $q_{0}=\frac{2N}{2N-qN+2q}$.
For any $\varepsilon>0$ there exists a positive $R_{\varepsilon}<\infty$ such that if $sup_{|y|<R_{\varepsilon}}\int_{B(y,r)}| u_{n}|^{2} dx \rightarrow 0 $ as $n \rightarrow \infty$.
Then
$lim_{n\rightarrow\infty}|u_{n}|^{q}_{L^{q}(g(x),\mathbb{R}^{N})}<\varepsilon$, for $1<q<2$.
\end{lemma}

\begin{proof}
For any $\varepsilon>0$, we can get from $H\ddot{o}lder$ inequality and the boundedness of $\{u_{n}\}$ that there exists a positive $R_{\varepsilon}<\infty$ such that
\begin{equation*}
     \int_{\mathbb{R}^{N}\backslash B(0,R_{\varepsilon})}g(x)|u_{n}|^{q}dx<\frac{\varepsilon}{2}.
\end{equation*}
From Lemma 3.5 we can see easily that if $sup_{|y|<R_{\varepsilon}}\int_{B(y,r)}| u_{n}|^{2} dx \rightarrow 0$, as $n \rightarrow \infty$,
\begin{equation*}
    lim_{n\rightarrow\infty} \int_{B(0,R_{\varepsilon})}g(x)|u_{n}|^{q}dx<\frac{\varepsilon}{2}.
\end{equation*}
So
\begin{equation*}
    lim_{n\rightarrow\infty} \int_{\mathbb{R}^{N}}g(x)|u_{n}|^{q}dx<\varepsilon.
\end{equation*}
\end{proof}

\begin{lemma}
Let $r>0$, $\{u_{n}\}$ is bounded in $H^{1}(\mathbb{R}^{N})$,
$h\in L^{\infty}(\mathbb{R}^{N})$ with $h(x)\geq0$, and $h(x)\rightarrow0$ as $|x|\rightarrow\infty$.
For any $\varepsilon>0$ there exists a positive $R_{\varepsilon}<\infty$ such that if $sup_{|y|<R_{\varepsilon}}\int_{B(y,r)}| u_{n}|^{2} dx \rightarrow 0 $ as $n \rightarrow \infty$.
Then
$lim_{n\rightarrow\infty}|u_{n}|^{p}_{L^{p}(h(x),\mathbb{R}^{N})}<\varepsilon$, for $2<p<2^{\ast}$.
\end{lemma}

\begin{proof}
For any $\varepsilon>0$, we can get from the boundedness of $\{u_{n}\}$ that there exists a positive $R_{\varepsilon}<\infty$ such that
\begin{equation*}
     \int_{\mathbb{R}^{N}\backslash B(0,R_{\varepsilon})}h(x)|u_{n}|^{p}dx<\frac{\varepsilon}{2}.
\end{equation*}
From the proof of Lemma 3.5 we can see that if $sup_{|y|<R_{\varepsilon}}\int_{B(y,r)}| u_{n}|^{2} dx \rightarrow 0$, as $n \rightarrow \infty$,
\begin{equation*}
    lim_{n\rightarrow\infty} \int_{B(0,R_{\varepsilon})}h(x)|u_{n}|^{p}dx<\frac{\varepsilon}{2}.
\end{equation*}
So
\begin{equation*}
    lim_{n\rightarrow\infty} \int_{\mathbb{R}^{N}}h(x)|u_{n}|^{p}dx<\varepsilon.
\end{equation*}
\end{proof}

Now we are able to prove Theorem 3.1 .

\begin{proof}
\\First let us verify the conditions :
\\$(B_{1})$ $a^{k}:=inf_{u\in Y^{k},\|u\|=\sigma_{k}}\varphi(u)\geq 0$,
\\$(B_{2})$ $b^{k}:=sup_{u\in Z^{k},\|u\|=s_{k}}\varphi(u)<0$,
\\$(B_{3})$ $d^{k}:=inf_{u\in Y^{k},\|u\|\leq \sigma_{k}}\varphi(u)\rightarrow 0,k\rightarrow\infty$.
\\
\\We write $u=y+z$, where $y \in Y ,z \in Z $. For every $u\in Y_{k}$, $y=u, z=0$.
\\From Lemma 3.1 and Lemma 3.2 we can see that
\begin{eqnarray*}
  \varphi(u) &=&\frac{1}{2}(\| y \|^{2}-\| z \|^{2})
    - \frac{1}{q}\int_{\mathbb{R}^{N}} g(x)|u | ^{q}dx
     +\frac{1}{p}\int_{\mathbb{R}^{N}} h(x)| u | ^{p}dx \\
   &=& \frac{1}{2}\parallel u \parallel^{2}
    - \frac{1}{q}| u|^{q}_{L^{q}(g(x),\mathbb{R}^{N})}
     +\frac{1}{p}  | u|^{p}_{L^{p}(h(x),\mathbb{R}^{N})}\\
   &\geq& \frac{1}{2}\| u \|^{2}
    - \frac{1}{q}| u|^{q}_{L^{q}(g(x),\mathbb{R}^{N})}\\
&\geq&  \frac{1}{2}\| u \|^{2}-\frac{1}{q}\beta^{q}_{k}\| u \|^{q}.
\end{eqnarray*}
Let $\sigma_{k}= (\frac{4\beta^{q}_{k}}{q})^{\frac{1}{2-q}}$,we get that
$$a^{k}:=inf_{u\in Y^{k},\|u\|=\sigma_{k}}\varphi(u)\geq 0,$$
and it is easy to see from Lemma 3.2 that $\sigma_{k}\rightarrow 0$ as $k\rightarrow\infty$.

Now for $u\in Z^{k}$, we use the Jensen inequality,
\begin{eqnarray*}
  \varphi(u) &=&\frac{1}{2}(\| y \|^{2}-\| z \|^{2})
    - \frac{1}{q}\int_{\mathbb{R}^{N}} g(x)| u | ^{q}dx
     +\frac{1}{p}\int_{\mathbb{R}^{N}} h(x)|u | ^{p}dx \\
   &=& \frac{1}{2}\| y \|^{2}-\frac{1}{2}\| z \|^{2}
    - \frac{1}{q}| u|^{q}_{L^{q}(h(x),\mathbb{R}^{N})}
     +\frac{1}{p}  | u|^{p}_{L^{p}(h(x),\mathbb{R}^{N})}\\
   &\leq& \frac{1}{2}\| y \|^{2}-\frac{1}{2}\| z \|^{2}
    - \frac{1}{q}| u|^{q}_{L^{q}(g(x),\mathbb{R}^{N})}
+\frac{2^{p-1}}{p}(| y|^{p}_{L^{p}(h(x),\mathbb{R}^{N})}+| z|^{p}_{L^{p}(h(x),\mathbb{R}^{N})}).
\end{eqnarray*}
Since the Sobolev space $H^{1}(\mathbb{R}^{N})$ embeds continuously in $L^{q}(g(x), \mathbb{R}^{N})$, we denote
$E_{k}$ the closure of $Z^{k}$ in $L^{q}(g(x),\mathbb{R}^{N})$, then there exists a continuous projection of $E_{k}$
on $\bigoplus_{j=0}^{k}e_{j}$, thus there exists a constant $C>0$ such that
$$| y|^{q}_{L^{q}(g(x),\mathbb{R}^{N})}<C| u|^{q}_{L^{q}(g(x),\mathbb{R}^{N})},$$
and in a finite-dimensional vector space all norms are equivalent, we have for some $C>0$
$$\| y \|^{q}<C | y|^{q}_{L^{q}(g(x),\mathbb{R}^{N})},$$
thus
\begin{eqnarray*}
  \varphi(u) &<& \frac{1}{2}\| y \|^{2}-\frac{1}{2}\| z \|^{2}
    - C\| y\|^{q}
+\frac{2^{p-1}}{p}(| y|^{p}_{L^{p}(h(x),\mathbb{R}^{N})}+| z|^{p}_{L^{p}(h(x),\mathbb{R}^{N})}) \\
   &<&  (\frac{1}{2}\|y\|^{2}-C\| y\|^{q} +C\| y\|^{p})
-\frac{1}{2}\| z\|^{2}  + C \| z\|^{p}.
\end{eqnarray*}
So we choose $s_{k}$ sufficiently small, it is easy to see
$$b^{k}:=sup_{u\in Z^{k},\|u\|=s_{k}}\varphi(u)<0.$$
\\
We know that for every $u\in Y_{k}$
$$ \varphi(u)\geq-\frac{1}{q}\beta_{k}\| u \|^{q}$$
and
$$\beta_{k} , \sigma_{k}\rightarrow 0  , k\rightarrow\infty$$
we get
$$d^{k}:=inf_{u\in Y^{k},\|u\|\leq \sigma_{k}}\varphi(u)\rightarrow 0,k\rightarrow\infty$$
thus condition$(B_{1})$,$(B_{2})$ and $(B_{3})$ of Theorem 2.1 is proved.
\\

Now let us show that any sequence $\{u_{n}\}$ such that  $\varphi(u_{n})\rightarrow c $ and $\varphi'(u_{n})\rightarrow 0$
is bounded in $H^{1}(\mathbb{R}^{N})$.

For $n$ big enough, we have
\begin{eqnarray*}
  \| u_{n}\| -c+1 &>&  \frac{1}{2}\langle \varphi'(u_{n}),u_{n}\rangle -\varphi(u_{n})\\
   &=&  (\frac{1}{2}-\frac{1}{p})\int_{\mathbb{R}^{N}}h(x)| u|^{p}dx
+(\frac{1}{q}-\frac{1}{2}) \int_{\mathbb{R}^{N}}g(x)| u|^{q}dx,
\end{eqnarray*}
thus
\begin{equation}\label{5}
    \int_{\mathbb{R}^{N}}h(x)| u|^{p}dx<C+\| u_{n}\|.
\end{equation}

\begin{equation*}
    \| y_{n}\| \geq \langle \varphi'(u_{n}),y_{n}\rangle
  =  \| y_{n}\|^{2}-\int_{\mathbb{R}^{N}}g(x)| u|^{q-2}u  y_{n} dx
+\int_{\mathbb{R}^{N}}h(x)| u|^{p-2}u y_{n} dx,
\end{equation*}
thus
\begin{equation*}
    \| y_{n}\|^{2}\leq \| y_{n}\|+\int_{\mathbb{R}^{N}}g(x)| u|^{q-1}  y_{n} dx
+ \int_{\mathbb{R}^{N}}h(x)| u|^{p-1} y_{n} dx.
\end{equation*}
Using $H\ddot{o}lder$ inequality and $(3.5)$ we have, for some $C>0$

\begin{eqnarray*}
  \| y_{n}\|^{2} & \leq& \| y_{n}\|+
|g(x)^{\frac{q-1}{q}}u_{n}^{q-1}|_{L^{\frac{q}{q-1}}} |g(x)^{\frac{1}{q}}y_{n}^{q-1}|_{L^{\frac{q}{q-1}}}
+|h(x)^{\frac{p-1}{p}}u_{n}^{p-1}|_{L^{\frac{p}{p-1}}} |h(x)^{\frac{1}{p}}y_{n}^{p-1}|_{L^{\frac{p}{p-1}}}\\
   &=& \| y_{n}\|+
| u_{n} |^{q-1}_{L^{q}(g(x),\mathbb{R}^{N})} | y_{n} |_{L^{q}(g(x),\mathbb{R}^{N})}
+[\int_{\mathbb{R}^{N}}h(x)|u|^{p} dx]^{\frac{p-1}{p}} |y_{n}|_{L^{p}(g(x),\mathbb{R}^{N})}\\
   &\leq& \|y_{n}\| + C\|u_{n}\|^{q-1}\| z_{n}\|
+ C(1+\|u_{n}\|)^{\frac{p-1}{p}} \| z_{n}\|\\
   &\leq&\| u_{n}\| + C\| u_{n}\|^{q}
+ C(1+\|u_{n}\|)^{\frac{p-1}{p}} \| u_{n}\|.
\end{eqnarray*}
Similarly we can get from $\| z
_{n}\| \geq -\langle \varphi'(u_{n}),z_{n}\rangle$ that
\begin{eqnarray*}
  \| z_{n}\|^{2} & \leq& \| z_{n}\|+
|g(x)^{\frac{q-1}{q}}u_{n}^{q-1}|_{L^{\frac{q}{q-1}}}|g(x)^{\frac{1}{q}}z_{n}^{q-1}|_{L^{\frac{q}{q-1}}}
+|h(x)^{\frac{p-1}{p}}u_{n}^{p-1}|_{L^{\frac{p}{p-1}}}|h(x)^{\frac{1}{p}}z_{n}^{p-1}|_{L^{\frac{p}{p-1}}}\\
   &\leq& \| u_{n}\| + C\| u_{n}\|^{q}
+ C (1+\| u_{n}\|)^{\frac{p-1}{p}} \| u_{n}\|.
\end{eqnarray*}
For $\|u_{n}\|^{2}=\| y_{n}\|^{2}+\| z_{n}\|^{2}$ ,
we have
$$  \| u_{n}\|^{2} \leq \| u_{n}\| + C\| u_{n}\|^{q}
+ C (1+\| u_{n}\|)^{\frac{p-1}{p}} \| u_{n}\|,$$
thus $\{u_{n}\}$ is bounded in $H^{1}(\mathbb{R}^{N})$.
\\

We write $u=y+z$ and $u_{n}=y_{n}+z_{n}$, where $y$, $y_{n}\in Y$ and $z$, $z_{n}\in Z$,
so that
\begin{center}
    $\langle \varphi'(u_{n})-\varphi'(u),y_{n}-y\rangle\rightarrow 0$  as $n\rightarrow\infty$.
\end{center}
By the boundedness of $\{u_{n}\}$, we may assume, up to a subsequence, that
\begin{center}
    $y_{n}\rightharpoonup y$ in $H^{1}(\mathbb{R}^{N}),$
\end{center}
\begin{center}
    $z_{n} \rightharpoonup z$ in $H^{1}(\mathbb{R}^{N}).$
\end{center}
\begin{eqnarray*}
  \langle \varphi'(u_{n})-\varphi'(u),y_{n}-y\rangle=\|y_{n}-y\|^{2} &+&
\int_{\mathbb{R}^{N}}g(x)(|u|^{q-2}u-|u_{n}|^{q-2}u_{n})(y_{n}-y)dx \\
   &-& \int_{\mathbb{R}^{N}}h(x)(|u|^{p-2}u-|u_{n}|^{p-2}u_{n})(y_{n}-y)dx.
\end{eqnarray*}

Using the $H\ddot{o}lder$ inequality we can get that
\begin{center}
    $y_{n} \rightarrow y$ in $H^{1}(\mathbb{R}^{N}),$
\end{center}
similarly we have
\begin{center}
    $z_{n}\rightarrow z$ in $H^{1}(\mathbb{R}^{N}),$
\end{center}
so
\begin{center}
    $u_{n}\rightarrow u$ in $H^{1}(\mathbb{R}^{N}).$
\end{center}
Thus the $(PS)_{c}$ condition holds for all $c\neq 0$.
and we get the conclusion we need from Corollary2.2.
\end{proof}
\\

In order to prove Theorem 3.2, we need only to show that the weak limit of $\{u_{n}\}$ is nontrivial.

\begin{proof}

Now let $\varepsilon=min \{\frac{2|c|q}{3(2-q)},\frac{2|c|p}{3(p-2)}\} $ and $\delta:=\overline{lim}_{n\rightarrow\infty}sup_{|y|<R_{\varepsilon}}
\int_{B(y,1)}| u_{n} |^{2}dx =0$
we get from Lemma 3.6 and Lemma 3.7 that
$$lim_{n\rightarrow\infty}\int_{\mathbb{R}^{N}}g(x)| u_{n}|^{q}dx<\frac{|c|}{3},$$
and
$$lim_{n\rightarrow\infty}\int_{\mathbb{R}^{N}}h(x)| u_{n}|^{p}dx<\frac{|c|}{3}.$$
So
\begin{eqnarray*}
  |c| &=& lim_{n\rightarrow\infty}|\varphi(u_{n})-\frac{1}{2}\langle\varphi'(u_{n}), u_{n}\rangle| \\
   &\leq&lim_{n\rightarrow\infty} |(\frac{1}{2}-\frac{1}{q})\int_{\mathbb{R}^{N}}g(x)| u_{n}|^{q}dx| +
lim_{n\rightarrow\infty}|(\frac{1}{2}-\frac{1}{p})\int_{\mathbb{R}^{N}}h(x)| u_{n}|^{q}dx|\\
   &\leq& \frac{2|c|}{3}.
\end{eqnarray*}
This is a contradiction. Thus $\delta>0$ and the weak limit of $\{u_{n}\}$ is nontrivial.
We can get the conclusion easily from the weakly sequentially continuity of $\varphi'$.

\end{proof}

\end{document}